\documentclass[final,3p,times]{elsarticle}
\usepackage{amsmath,amssymb,amsfonts,epsf,amsthm}
\usepackage{hyperref}
\hypersetup{colorlinks=true, urlcolor=blue, citecolor=cyan, pdfborder={0 0 0},}
\usepackage{soul}
\usepackage{natbib}
\usepackage{multirow}
\usepackage{amssymb}
\usepackage{amsmath,amssymb,amsfonts,epsf,amsthm}
\usepackage{hyperref}
\hypersetup{colorlinks=true, urlcolor=blue, citecolor=cyan, pdfborder={0 0 0},}
\usepackage{soul}
\usepackage{natbib}
\usepackage{multirow}
\usepackage{amsmath,amssymb,amscd,amsthm,verbatim,alltt,amsfonts,array}
\usepackage[english]{babel}
\usepackage{latexsym}
\usepackage{amssymb}
\usepackage{euscript}
\usepackage{graphicx}
\usepackage{amssymb}
\usepackage{pgf,tikz}
\usetikzlibrary{arrows}
\usepackage{mathrsfs}
\usepackage[ruled, vlined, Algorithm]{algorithm}
\usepackage{algorithmic}

\newtheorem{theorem}{Theorem}[section]
\newtheorem{lemma}{Lemma}[section]

\newtheorem{conclusion}{Conclusion}[section]

\newtheorem{remark}{Remark}[section]
\newtheorem{example}{Example}[section]

\newtheorem{definition}{Definition}[section]

\numberwithin{equation}{section}
\definecolor{ffqqff}{rgb}{1.,0.,1.}
\definecolor{dcrutc}{rgb}{0.8627450980392157,0.0784313725490196,0.23529411764705882}
\definecolor{qqqqff}{rgb}{0.,0.,1.}
\definecolor{ffqqqq}{rgb}{1.,0.,0.}
\definecolor{sqsqsq}{rgb}{0.12549019607843137,0.12549019607843137,0.12549019607843137}
\journal{--}
\allowdisplaybreaks[1]
\begin{document}
\begin{frontmatter}
\title{    Metric dimension, doubly resolving set and strong metric dimension for $(C_n\Box P_k)\Box P_m$ }

\tnotetext[label1]{}
\author[label1]{Jia-Bao Liu}
\ead{liujiabaoad@163.com;liujiabao@ahjzu.edu.cn}
\author[label2]{Ali Zafari \corref{1}}
\ead{zafari.math.pu@gmail.com; zafari.math@pnu.ac.ir}
\address[label1]{School of Mathematics and Physics, Anhui Jianzhu University, Hefei 230601, P.R. China}
\address[label2]{Department of Mathematics, Faculty of Science, Payame Noor University, P.O. Box 19395-4697, Tehran, Iran}
\cortext[1]{Corresponding author}
\begin{abstract}
A subset $Q = \{q_1, q_2, ..., q_l\}$ of vertices of a connected graph $G$ is a doubly resolving set of $G$ if for any various vertices $x, y \in V(G)$ we have $r(x|Q)-r(y|Q)\neq\lambda I$, where $\lambda$ is an  integer, and $I$ indicates  the unit $l$- vector $(1,..., 1)$. A doubly resolving set of vertices of graph $G$  with the minimum size, is denoted by $\psi(G)$.
In this work, we will consider the computational study of some resolving sets with the minimum size for $(C_n\Box P_k)\Box P_m$.
\end{abstract}
\begin{keyword}
cartesian product,  resolving set, doubly resolving set, strong resolving set
\MSC[2020]  05C12, 05C76.
\end{keyword}
\end{frontmatter}
\section{Introduction and Preliminaries}
\label{sec:introduction}
All graphs considered in this work are assumed to be finite and connected. We use $d_G(p, q)$  to indicate the distance between two vertices $p$ and $q$ in graph $G$ as the length of a shortest path between $p$ and $q$ in $G$, see
~\cite{pap-cg-1}.
The cartesian  product of two graphs $G$ and $H$, denoted by $G \Box H$, is the graph with vertex set $V(G) \times V(H)$ and with edge set
$E(G \times H)$ so that $(g_1,h_1)(g_2,h_2) \in E(G \Box H)$, whenever $h_1 = h_2$ and $g_1g_2 \in E(G)$, or $g_1 = g_2$ and $h_1h_2 \in E(H)$.
A graphical representation of a vertex $p$ of a connected graph $G$ relative to an arranged subset $Q = \{q_1, ..., q_k\}$ of vertices of $G$ is defined as the $k$-tuple $(d(p, q_1), ..., d(p, q_k ))$, and this $k$-tuple is denoted by $r(p|Q)$, where $d(p, q_i)$ is considered as the minimum length of a shortest path from $p$ to $q_i$ in graph $G$. If any vertices  $p$ and $q$ that belong to  $V(G)-Q$  have various representations with respect to the set $Q$, then $Q$ is called a resolving set for $G$
 ~\cite{pap-psb.gh1-2}.
Slater
~\cite{pap-pjs-1}
considered the concept and notation of the metric dimension problem under the term locating set. Also, Harary and Melter
~\cite{pap-fh-1}
considered these problems under the term metric dimension as follows.
A resolving set of vertices of graph $G$ with the minimum size or cardinality is called the metric dimension of $G$ and this minimum size denoted by $\beta(G)$. Resolving parameters in graphs have been studied in
~\cite{pap-M.Abas .ET.AL,pap-M.Ali.ET.AL,pap-M.Baca.ET.AL,pap-mj-1,pap-J.B.Liu.ET.AL,pap-jb1-1,pap-jb1-2,pap-J.-B.Liu.ET.AL.2,pap-X.Zhang.ET.AL}.\\

In 2007 C\'{a}ceres et al.
~\cite{pap-jc-1}
 considered the concept and notation of a doubly resolving set of graph $G$.
Two vertices $u$, $v$ in a graph $G$ are doubly resolved by $x, y \in V (G)$ if $d(u,x)-d(u,y)\neq d(v,x)-d(v,y)$,
and we can see that a subset $Q = \{q_1, q_2, ..., q_l\}$ of vertices of a graph $G$ is a doubly resolving set of $G$ if for any various vertices $x, y \in V(G)$ we have $r(x|Q)-r(y|Q)\neq\lambda I$, where $\lambda$ is an  integer, and $I$ indicates  the unit $l$- vector $(1,..., 1)$, see
~\cite{pap-aa.s.s1-2}.
 A doubly resolving set of vertices of graph $G$  with the minimum size, is denoted by $\psi(G)$.
In 2000 Chartrand et al. showed that for every connected graph $G$ and the path $P_2$, $\beta(G \Box P_2)\leq  \beta(G)+1$, see Theorem 7 in ~\cite{pap-chr-1}.
In 2007 C\'{a}ceres et al. obtained an upper bound for the metric dimension of cartesian product of graphs $G$ and $H$. They showed  that for all graphs $G$ and $H\neq K_1$, $\beta(G \Box H)\leq  \beta(G)+\psi(H)-1$.
In particular,  C\'{a}ceres et al.
~\cite{pap-jc-1}
showed that for every connected graph $G$ and the path $P_n$, $\beta(G \Box P_n)\leq \beta(G)+1$.
Doubly resolving sets have played a special role in the study of resolving sets.
Applications of above concepts  and related parameters in graph theory and other sciences have a long history, in particular, if we consider a graph as a chemical compound then the determination of a doubly resolving set with the minimum size is very useful to analysis  of  chemical compound, and note that these problems are NP hard, see
~\cite{pap-M.Ahmad.ET.AL,pap-pj.jh-1,pap-X. Chen1-2,pap-J.Kratica.ET.AL1,pap-J.Kratica.ET.AL,pap-Khuller}.\\

The notion of a strong metric dimension problem set of vertices of graph $G$ introduced by A. Seb\"{o} and E. Tannier
~\cite{pap-as-1},
indeed introduced a more restricted invariant than the metric dimension and this was further investigated by O. R. Oellermann and Peters-Fransen ~\cite{pap-oro-1}.
A set $Q\subseteq V(G)$  is called  strong resolving set of $G$, if for any various vertices $p$ and $q$ of $G$ there is a vertex of $Q$, say $r$  so that $p$ belongs to a shortest $q - r$ path or $q$ belongs to a shortest $p - r$ path. A strong metric basis of $G$ is indicated by $sdim(G)$  defined as the minimum size of a strong resolving set of $G$.\\

Now, we use $C_n$ and $P_k$ to denote the cycle on $n \geq 3$ and the path on  $k\geq 3$ vertices, respectively.
In this article, we will consider the computational study of some resolving sets with the minimum size for $(C_n\Box P_k)\Box P_m$.
Indeed, in Section 3, we define a graph isomorphic to the cartesian product $C_n\Box P_k$, and we will consider the determination of a doubly resolving set of vertices with the minimum size for the cartesian product $C_n\Box P_k$. In particular, in  Section 3, we construct a graph  so that this graph is isomorphic to $(C_n\Box P_k)\Box P_m$, and we compute some resolving parameters  with the minimum size for
$(C_n\Box P_k)\Box P_m$.\newline
\section{ Some Facts}
\begin{definition}\label{b.1}
Consider two graphs $G$ and $H$. If there is a bijection, $\theta: V(G)\rightarrow V(H)$  so that $u$ is adjacent to $v$ in $G$ if and only if $\theta(u)$ is adjacent to $\theta(v)$ in $H$, then we say that $G$ and $H$ are  isomorphic.
\end{definition}
\begin{definition}\label{b.1.1}
A vertex $u$ of a graph $G$ is called maximally distant from a vertex $v$ of $G$, if for every $w\in N_G(u)$, we have  $d(v,w) \leq d(v, u)$, where $N_G (u)$ to denote the set of neighbors that $u$ has in $G$. If $u$ is maximally distant from $v$ and $v$ is maximally distant from $u$, then $u$ and $v$ are said to be mutually maximally distant.
\end{definition}
 \begin{remark} \label{b.4}
Suppose that $n$ is an even natural number  greater than or equal to $4$ and  $G$ is the cycle graph $C_n$ . Then $\beta(G) =2$, $\psi(G)=3$ and  $sdim(G)=\lceil \frac{n}{2}\rceil$.
\end{remark}
\begin{remark} \label{b.5}
Suppose that $n$ is an odd natural number  greater than or equal to $3$ and  $G$ is the cycle graph $C_n$. Then $\beta(G) =2$, $\psi(G)=2$ and  $sdim(G)=\lceil \frac{n}{2}\rceil$.
\end{remark}
\begin{remark} \label{b.5.1}
Consider the path $P_n$ for each $n\geq2$. Then $\beta(P_n) =1$, $\psi(P_n)=2$.
\end{remark}
\begin{theorem}\label{b.6}
Suppose that $n$ is an odd integer  greater than or equal to $3$.  Then the  minimum size of a resolving set in the cartesian product
$C_n\Box P_k$ is  $2$.
\end{theorem}
\begin{theorem}\label{b.7}
Suppose that $n$ is an even integer  greater than or equal to $4$.
Then the  minimum size of a resolving set in the cartesian product $C_n\Box P_k$ is  $3$.
\end{theorem}
\begin{theorem}\label{b.8}
If $n$ is an even or odd integer is greater than or equal to $3$, then the minimum size of a strong resolving set in the cartesian product
$C_n\Box P_k$ is  $n$.
\end{theorem}
\section{Main Results}

Some resolving parameters such as the minimum size of resolving sets and strong resolving sets calculated for
the cartesian  product $C_n\Box P_k$, see
~\cite{pap-jc-1, pap-ja-1},
but in this section we will determine some resolving sets of vertices with the minimum size for $(C_n\Box P_k)\Box P_m$. Suppose  $n$ and $k$ are  natural numbers  greater than or equal to $3$, and $[n]=\{1, ..., n\}$. Now, suppose that $G$ is a graph with vertex set $\{x_1,  ..., x_{nk}\}$ on layers $V_1, V_2, ..., V_k$, where $V_p=\{x_{(p-1)n+1}, x_{(p-1)n+2}, ..., x_{(p-1)n+n}\}$ for $1\leq p\leq k$, and the edge set of graph $G$ is
$E(G)=\{x_ix_j\,|\, x_i, x_j\in V_p,  1\leq i<j\leq nk, j-i=1 \text{or}    j-i=n-1  \}\cup \{x_ix_j\,|\, x_i\in V_q, x_j\in V_{q+1},  1\leq i<j\leq nk,  1\leq q\leq k-1, j-i=n  \}$.
We can see that this graph is isomorphic to the cartesian product $C_n\Box P_k$. So, we can assume throughout this article
$V(C_n\Box P_k)=\{x_1,  ..., x_{nk}\}$. Now, in this section we give a more elaborate description of the cartesian product $C_n\Box P_k$,
that are required to prove of Theorems. We use $V_p$, $1\leq p\leq k$, to indicate a layer of the cartesian product $C_n\Box P_k$, where $V_p$, is defined already. Also, for $1\leq e<d\leq nk$, we say that two vertices $x_e$ and  $x_d$ in  $C_n\Box P_k$ are compatible, if $n|d-e$.
We can see that the degree of a vertex in the layers $V_1$ and $V_k$ is 3, also the degree of a vertex in the layer $V_p$, $1<p<k$ is 4, and hence
$C_n\Box P_k$ is not regular. We say that two layers of $C_n\Box P_k$ are congruous, if the degree of compatible vertices in two layers are identical.
Note that, if $n$ is an even natural  number, then  $C_n\Box P_k$ contains no cycles of odd length, and hence in this case $C_n\Box P_k$ is bipartite.
For more result of families of graphs with constant metric, see
~\cite{pap-M.Ahmad.ET.AL,pap-mi-1}.
The cartesian product $C_4\Box P_3$ is depicted in Figure 1.\newline
\begin{center}
\begin{tikzpicture}[line cap=round,line join=round,>=triangle 45,x=15.0cm,y=15.0cm]
\clip(6.199538042905856,0.29983666911777535) rectangle (6.450634938365554,0.7288631527509081);
\draw [color=black] (6.298712432261516,0.6857056072380374)-- (6.248939348759354,0.6177720743499512);
\draw [color=black] (6.248939348759354,0.6177720743499512)-- (6.348485515763677,0.6197899020594984);
\draw [color=black] (6.348485515763677,0.6197899020594984)-- (6.398258599265839,0.6863782164745531);
\draw [color=black] (6.398258599265839,0.6863782164745531)-- (6.298712432261516,0.6857056072380374);
\draw [color=black] (6.3,0.55)-- (6.246921521049807,0.4859406639928731);
\draw [color=black] (6.246921521049807,0.4859406639928731)-- (6.350029890846933,0.4852492773545749);
\draw [color=black] (6.350029890846933,0.4852492773545749)-- (6.4,0.55);
\draw [color=black] (6.4,0.55)-- (6.3,0.55);
\draw [color=black] (6.298039823025,0.4328045343081325)-- (6.25025629000202,0.3680152963618003);
\draw [color=black] (6.25025629000202,0.3680152963618003)-- (6.350861337520641,0.3671838496880927);
\draw [color=black] (6.350861337520641,0.3671838496880927)-- (6.39960381773887,0.4328045343081325);
\draw [color=black] (6.39960381773887,0.4328045343081325)-- (6.298039823025,0.4328045343081325);
\draw [color=black] (6.248939348759354,0.6177720743499512)-- (6.25025629000202,0.3680152963618003);
\draw [color=black] (6.348485515763677,0.6197899020594984)-- (6.350861337520641,0.3671838496880927);
\draw [color=black] (6.398258599265839,0.6863782164745531)-- (6.39960381773887,0.4328045343081325);
\draw [color=black] (6.298712432261516,0.6857056072380374)-- (6.298039823025,0.4328045343081325);
\draw (6.204322963708524,0.34972346954023265) node[anchor=north west] {Figure 1. $C_4\Box P_3$};
\begin{scriptsize}
\draw [fill=black] (6.298712432261516,0.6857056072380374) circle (1.5pt);
\draw[color=black] (6.304300323793015,0.6972681791500184) node {$x_{12}$};
\draw [fill=black] (6.248939348759354,0.6177720743499512) circle (1.5pt);
\draw[color=black] (6.237784589896406,0.6207750851689172) node {$x_9$};
\draw [fill=black] (6.348485515763677,0.6197899020594984) circle (1.5pt);
\draw[color=black] (6.343378317457272,0.6324153386008239) node {$x_{10}$};
\draw [fill=black] (6.398258599265839,0.6863782164745531) circle (1.5pt);
\draw[color=black] (6.4040739246379275,0.698099625823726) node {$x_{11}$};
\draw [fill=black] (6.3,0.55) circle (1.5pt);
\draw[color=black] (6.31596321714043,0.5617423713356761) node {$x_8$};
\draw [fill=black] (6.246921521049807,0.4859406639928731) circle (1.5pt);
\draw[color=black] (6.235290249875283,0.49439519076535876) node {$x_5$};
\draw [fill=black] (6.350029890846933,0.4852492773545749) circle (1.5pt);
\draw[color=black] (6.337558190741319,0.4968895307864816) node {$x_6$};
\draw [fill=black] (6.4,0.55) circle (1.5pt);
\draw[color=black] (6.409636817985343,0.5617423713356761) node {$x_7$};
\draw [fill=black] (6.298039823025,0.4328045343081325) circle (1.5pt);
\draw[color=black] (6.3134688771193075,0.4445083903429015) node {$x_4$};
\draw [fill=black] (6.25025629000202,0.3680152963618003) circle (1.5pt);
\draw[color=black] (6.239447483243821,0.37217252973033843) node {$x_1$};
\draw [fill=black] (6.350861337520641,0.3671838496880927) circle (1.5pt);
\draw[color=black] (6.337558190741319,0.3788241031199994) node {$x_2$};
\draw [fill=black] (6.39960381773887,0.4328045343081325) circle (1.5pt);
\draw[color=black] (6.409062604680173,0.4386882636269481) node {$x_3$};
\end{scriptsize}
\end{tikzpicture}
\end{center}
Now, let  $m\geq2$ be an  integer. Suppose $1\leq i\leq m$ and consider $i^{th}$ copy of the cartesian product $C_n\Box P_k$ with the vertex set  $\{x_1^{(i)},  ..., x_{nk}^{(i)}\}$ on the layers $V_1^{(i)}, V_2^{(i)}, ..., V_k^{(i)}$, where it can be defined $V_p^{(i)}$ as similar $V_p$ on the vertex set $\{x_1^{(i)},  ..., x_{nk}^{(i)}\}$. Also, suppose that $K$ is a graph with vertex set $\{x_1^{(1)},  ..., x_{nk}^{(1)}\}\cup \{x_1^{(2)},  ..., x_{nk}^{(2)}\}\cup ... \cup\{x_1^{(m)},  ..., x_{nk}^{(m)}\}$ so that for $1\leq t \leq nk$, the vertex $x_t^{(r)}$ is adjacent to the vertex $x_t^{(r+1)}$ in $K$, for $1\leq r\leq m-1$, then we can see that the graph $K$ is isomorphic to $(C_n\Box P_k)\Box P_m$.
For $1\leq e<d\leq nk$, we say that two vertices $x_e^{(i)}$ and  $x_d^{(i)}$ in $i^{th}$ copy of the cartesian product $C_n\Box P_k$ are compatible, if $n|d-e$.
The graph $(C_4\Box P_3)\Box P_2$ is depicted in Figure 2. \newline
\begin{center}
\begin{tikzpicture}[line cap=round,line join=round,>=triangle 45,x=4cm,y=4cm]
\clip(6.5645454339715394,1.0260516182926602) rectangle (8.674716775107852,3.529151416054353);
\draw (8.218853014490843,3.2967058297352008)-- (7.204885316407781,3.292391073573145);
\draw (8.218853014490843,3.2967058297352008)-- (8.2,2.6);
\draw (8.2,2.6)-- (7.2,2.6);
\draw (7.2,2.6)-- (7.204885316407781,3.292391073573145);
\draw (6.989147508305002,2.7357875286679763)-- (7.998800450226008,2.7530465533161985);
\draw (7.998800450226008,2.7530465533161985)-- (7.994485694063952,2.1403511783043068);
\draw (7.994485694063952,2.1403511783043068)-- (6.997777020629113,2.148980690628418);
\draw (6.997777020629113,2.148980690628418)-- (6.989147508305002,2.7357875286679763);
\draw (7.204885316407781,3.292391073573145)-- (6.989147508305002,2.7357875286679763);
\draw (8.218853014490843,3.2967058297352008)-- (7.998800450226008,2.7530465533161985);
\draw (8.2,2.6)-- (7.994485694063952,2.1403511783043068);
\draw (7.2,2.6)-- (6.997777020629113,2.148980690628418);
\draw (6.997777020629113,2.148980690628418)-- (7.,1.4);
\draw (7.,1.4)-- (8.,1.4);
\draw (8.2,2.6)-- (8.2,2.);
\draw (7.2,2.6)-- (7.2,2.);
\draw (8.2,2.)-- (8.,1.4);
\draw (8.2,2.)-- (7.2,2.);
\draw (7.2,2.)-- (7.,1.4);
\draw (7.994485694063952,2.1403511783043068)-- (8.,1.4);
\draw [color=black] (7.739915080502673,3.089597533956533)-- (7.398986742414686,3.09282282304291);
\draw [color=black] (7.398986742414686,3.09282282304291)-- (7.307807711449688,2.889423446274836);
\draw [color=black] (7.739915080502673,3.089597533956533)-- (7.657934713423617,2.8911187505019766);
\draw [color=black] (7.307807711449688,2.889423446274836)-- (7.657934713423617,2.8911187505019766);
\draw [color=black] (7.748544592826784,2.455328378134364)-- (7.399049343700282,2.455328378134364);
\draw [color=black] (7.399049343700282,2.455328378134364)-- (7.34295751359356,2.2568495946798075);
\draw [color=black] (7.748544592826784,2.455328378134364)-- (7.679508494233895,2.261164350841863);
\draw [color=black] (7.679508494233895,2.261164350841863)-- (7.34295751359356,2.2568495946798075);
\draw [color=black] (7.399049343700282,1.8512625154465834)-- (7.7528593489888395,1.855577271608639);
\draw [color=black] (7.7528593489888395,1.855577271608639)-- (7.692452762720062,1.6053214142094154);
\draw [color=black] (7.692452762720062,1.6053214142094154)-- (7.347272269755615,1.6010066580473599);
\draw [color=black] (7.399049343700282,1.8512625154465834)-- (7.347272269755615,1.6010066580473599);
\draw [color=black] (7.307807711449688,2.889423446274836)-- (7.34295751359356,2.2568495946798075);
\draw [color=black] (7.398986742414686,3.09282282304291)-- (7.399049343700282,2.455328378134364);
\draw [color=black] (7.739915080502673,3.089597533956533)-- (7.748544592826784,2.455328378134364);
\draw [color=black] (7.657934713423617,2.8911187505019766)-- (7.679508494233895,2.261164350841863);
\draw [color=black] (7.679508494233895,2.261164350841863)-- (7.692452762720062,1.6053214142094154);
\draw [color=black] (7.748544592826784,2.455328378134364)-- (7.7528593489888395,1.855577271608639);
\draw [color=black] (7.34295751359356,2.2568495946798075)-- (7.347272269755615,1.6010066580473599);
\draw [color=black] (7.399049343700282,2.455328378134364)-- (7.399049343700282,1.8512625154465834);
\draw [color=black] (7.398986742414686,3.09282282304291)-- (7.204885316407781,3.292391073573145);
\draw [color=black] (7.739915080502673,3.089597533956533)-- (8.218853014490843,3.2967058297352008);
\draw [color=black] (7.657934713423617,2.8911187505019766)-- (7.998800450226008,2.7530465533161985);
\draw [color=black] (7.307807711449688,2.889423446274836)-- (6.989147508305002,2.7357875286679763);
\draw [color=black] (7.748544592826784,2.455328378134364)-- (8.2,2.6);
\draw [color=black] (7.399049343700282,2.455328378134364)-- (7.2,2.6);
\draw [color=black] (7.34295751359356,2.2568495946798075)-- (6.997777020629113,2.148980690628418);
\draw [color=black] (7.679508494233895,2.261164350841863)-- (7.994485694063952,2.1403511783043068);
\draw [color=black] (7.7528593489888395,1.855577271608639)-- (8.2,2.);
\draw [color=black] (7.399049343700282,1.8512625154465834)-- (7.2,2.);
\draw [color=black] (7.347272269755615,1.6010066580473599)-- (7.,1.4);
\draw [color=black] (7.692452762720062,1.6053214142094154)-- (8.,1.4);
\draw (7.304061064625124,1.2297922995058213) node[anchor=north west] {Figure 2. $(C_4\Box P_3)\Box P_2$};
\begin{scriptsize}
\draw [fill=black] (8.218853014490843,3.2967058297352008) circle (1.5pt);
\draw[color=black] (8.25268250688059,3.3642184836436986) node {$x_{11}^{(2)}$};
\draw [fill=black] (7.204885316407781,3.292391073573145) circle (1.5pt);
\draw[color=black] (7.238830069415096,3.3593675150433855) node {$x_{12}^{(2)}$};
\draw [fill=black] (8.2,2.6) circle (1.5pt);
\draw[color=black] (8.25268250688059,2.631722224996382) node {$x_7^{(2)}$};
\draw [fill=black] (7.2,2.6) circle (1.5pt);
\draw[color=black] (7.136959728808516,2.636573193596695) node {$x_8^{(2)}$};
\draw [fill=black] (6.989147508305002,2.7357875286679763) circle (1.5pt);
\draw[color=black] (6.9235171103947275,2.7626983772048423) node {$x_9^{(2)}$};
\draw [fill=black] (7.998800450226008,2.7530465533161985) circle (1.5pt);
\draw[color=black] (8.078643762868055,2.7626983772048423) node {$x_{10}^{(2)}$};
\draw [fill=black] (7.994485694063952,2.1403511783043068) circle (1.5pt);
\draw[color=black] (8.063494731468367,2.146625364965046) node {$x_6^{(2)}$};
\draw [fill=black] (6.997777020629113,2.148980690628418) circle (1.5pt);
\draw[color=black] (6.913815173194101,2.1563273021656726) node {$x_5^{(2)}$};
\draw [fill=black] (7.,1.4) circle (1.5pt);
\draw[color=black] (6.942026796397861,1.3559174831139686) node {$x_1^{(2)}$};
\draw [fill=black] (8.,1.4) circle (1.5pt);
\draw[color=black] (8.044090857067115,1.3704703889149088) node {$x_2^{(2)}$};
\draw [fill=black] (8.2,2.) circle (1.5pt);
\draw[color=black] (8.257533475480901,2.030202118557525) node {$x_3^{(2)}$};
\draw [fill=black] (7.2,2.) circle (1.5pt);
\draw[color=black] (7.127257791607889,2.0447550243584653) node {$x_4^{(2)}$};
\draw [fill=black] (7.739915080502673,3.089597533956533) circle (1.5pt);
\draw[color=black] (7.7384798352473725,3.1653287710308513) node {$x_{11}^{(1)}$};
\draw [fill=black] (7.398986742414686,3.09282282304291) circle (1.5pt);
\draw[color=black] (7.432868813427631,3.1604778024305378) node {$x_{12}^{(1)}$};
\draw [fill=black] (7.307807711449688,2.889423446274836) circle (1.5pt);
\draw[color=black] (7.268532006615723,2.942184215416437) node {$x_9^{(1)}$};
\draw [fill=black] (7.657934713423617,2.8911187505019766) circle (1.5pt);
\draw[color=black] (7.626907557440165,2.9518861526170634) node {$x_{10}^{(1)}$};
\draw [fill=black] (7.748544592826784,2.455328378134364) circle (1.5pt);
\draw[color=black] (7.802138552650193,2.5250009157894877) node {$x_7^{(1)}$};
\draw [fill=black] (7.399049343700282,2.455328378134364) circle (1.5pt);
\draw[color=black] (7.452868813427631,2.5250009157894877) node {$x_8^{(1)}$};
\draw [fill=black] (7.34295751359356,2.2568495946798075) circle (1.5pt);
\draw[color=black] (7.277637818217603,2.3067073287753868) node {$x_5^{(1)}$};
\draw [fill=black] (7.679508494233895,2.261164350841863) circle (1.5pt);
\draw[color=black] (7.627205620239538,2.321260234576327) node {$x_6^{(1)}$};
\draw [fill=black] (7.399049343700282,1.8512625154465834) circle (1.5pt);
\draw[color=black] (7.452868813427631,1.9186298407503182) node {$x_4^{(1)}$};
\draw [fill=black] (7.7528593489888395,1.855577271608639) circle (1.5pt);
\draw[color=black] (7.806989521250506,1.9283317779509448) node {$x_3^{(1)}$};
\draw [fill=black] (7.692452762720062,1.6053214142094154) circle (1.5pt);
\draw[color=black] (7.757883709648626,1.6421246305324568) node {$x_2^{(1)}$};
\draw [fill=black] (7.347272269755615,1.6010066580473599) circle (1.5pt);
\draw[color=black] (7.2921907240185435,1.6518265677330835) node {$x_1^{(1)}$};
\end{scriptsize}
\end{tikzpicture}
\end{center}
\begin{theorem}\label{d.1}
Consider the cartesian product $C_n\Box P_k$. If $n\geq3$ is an odd integer, then the minimum size of a doubly resolving set of vertices for the cartesian product $C_n\Box P_k$ is $3$.
\end{theorem}
\begin{proof}
In the following cases we show that the minimum size of a doubly resolving set of vertices for the cartesian product $C_n\Box P_k$ is $3$.\\

Case 1. First, we show that the minimum size of a doubly resolving set of vertices in $C_n\Box P_k$ must be greater than $2$.
Consider the cartesian product $C_n\Box P_k$ with the vertex set  $\{x_1,  ..., x_{nk}\}$ on the layers
$V_1, V_2, ..., V_k$, is defined already. Based on  Theorem \ref{b.6}, we know that $\beta(C_n\Box P_k)=2$.
We can show that if $n$ is an odd integer then all the elements of every minimum resolving set of  vertices in $C_n\Box P_k$ must be lie in
exactly one of the congruous layers $V_1$ or $V_k$. Without lack of theory if we consider the layer $V_1$ of the cartesian product
$C_n\Box P_k$ then  we can show that all the minimum resolving sets of vertices in the layer $V_1$ of $C_n\Box P_k$ are the sets as to form
$M_i=\{x_i, x_{\lceil\frac{n}{2}\rceil+i-1}\}$, $1\leq i\leq \lceil\frac{n}{2}\rceil$ and $N_j=\{x_j, x_{\lceil\frac{n}{2}\rceil+j}\}$,  $1\leq j\leq \lfloor\frac{n}{2}\rfloor$.
On the other hand, we can see that the arranged subsets $M_i$ of vertices in $C_n\Box P_k$ cannot be doubly resolving sets for $C_n\Box P_k$ because for
$1\leq i\leq \lceil\frac{n}{2}\rceil$ and  two compatible vertices $x_{i+n}$ and $x_{i+2n}$ with respect to $x_i$, we have
$r(x_{i+n}|M_i)-r(x_{i+2n}|M_i)=-I$, where $I$ indicates the unit $2$-vector $(1, 1)$. By applying the same argument we can show that
the arranged subsets $N_j$ of vertices in $C_n\Box P_k$ cannot be  doubly resolving sets for $C_n\Box P_k$. Hence, the minimum size of a doubly resolving set in $C_n\Box P_k$
must be greater than $2$.\\

Case 2. Now, we show that the minimum size of a doubly resolving set of vertices in  the cartesian product $C_n\Box P_k$ is $3$.
For $1\leq i\leq \lceil\frac{n}{2}\rceil$, let $x_i$ be a  vertex in the layer $V_1$ of $C_n\Box P_k$ and $x_c$ be a compatible
vertex with  respect to $x_i$, where $x_c$ lie in the layer $V_k$ of $C_n\Box P_k$, then we can show that
the arranged subsets $A_i=M_i\cup x_c=\{x_i, x_{\lceil\frac{n}{2}\rceil+i-1}, x_c\}$ of vertices in  the cartesian product $C_n\Box P_k$ are
doubly resolving sets with the minimum size for the cartesian product $C_n\Box P_k$. It will be enough to show that for any compatible vertices $x_e$
and $x_d$ in $C_n\Box P_k$, $r(x_e|A_i)-r(x_d|A_i)\neq\lambda I$. Suppose  $x_e\in V_p$ and $x_d\in V_q$ are compatible vertices in the
cartesian product $C_n\Box P_k$, $1\leq p<q\leq k$. Hence, $r(x_e|M_i)-r(x_d|M_i)=-\lambda I$, where $\lambda$ is a positive integer, and $I$
indicates the unit $2$-vector $(1, 1)$. Also, for the compatible vertex $x_c$ with  respect to $x_i$,  $r(x_e|x_c)-r(x_d|x_c)=\lambda$. So, $r(x_e|A_i)-r(x_d|A_i)\neq\lambda I$,  where $I$ indicates the unit $3$-vector $(1, 1, 1)$. Especially, for $1\leq j\leq \lfloor\frac{n}{2}\rfloor$
if we consider the arranged subsets $B_j=N_j\cup x_c=\{x_j, x_{\lceil\frac{n}{2}\rceil+j}, x_c\}$ of vertices in the cartesian product $C_n\Box P_k$,
where $x_c$ lie in the layer $V_k$ of the cartesian product $C_n\Box P_k$ and $x_c$ is a compatible vertex with  respect to $x_j$, then by applying
the same argument we can show that the arranged subsets $B_j=N_j\cup x_c=\{x_j, x_{\lceil\frac{n}{2}\rceil+j}, x_c\}$ of vertices in the cartesian product $C_n\Box P_k$  are doubly  resolving sets with the minimum size
for the cartesian product $C_n\Box P_k$.\newline
\end{proof}
\begin{theorem}\label{f.1}
Consider the cartesian product $C_n\Box P_k$. If $n\geq3$ is an odd integer, then the  minimum size of a resolving set of vertices in $(C_n\Box P_k)\Box P_2$ is  $3$.
\end{theorem}
\begin{proof}
Suppose $V((C_n\Box P_k)\Box P_2)=\{x_1^{(1)},  ..., x_{nk}^{(1)}\}\cup \{x_1^{(2)},  ..., x_{nk}^{(2)}\}$. Based on Theorem \ref{b.6}, we know that if $n\geq3$ is an odd integer, then the  minimum size of a resolving set of vertices in $C_n\Box P_k$ is $2$. Also, by definition of $(C_n\Box P_k)\Box P_2$ we can verify that for $1\leq t \leq nk$, the vertex $x_t^{(1)}$ is adjacent to the vertex $x_t^{(2)}$ in $(C_n\Box P_k)\Box P_2$, and hence none of the minimal resolving sets of $C_n\Box P_k$ cannot be a resolving set for $(C_n\Box P_k)\Box P_2$. Therefore, the  minimum size of a resolving set of vertices in $(C_n\Box P_k)\Box P_2$ must be greater than $2$.
Now, we show that the  minimum size of a resolving set of vertices in $(C_n\Box P_k)\Box P_2$ is  $3$.
For $1\leq i\leq \lceil\frac{n}{2}\rceil$, let $x_i^{(1)}$ be a  vertex in the layer $V_1^{(1)}$ of $(C_n\Box P_k)\Box P_2$ and $x_c^{(1)}$ be a compatible  vertex with  respect to $x_i^{(1)}$, where $x_c^{(1)}$ lie in the layer $V_k^{(1)}$ of $(C_n\Box P_k)\Box P_2$. Based on Theorem \ref{d.1}, we know that for $1\leq i\leq \lceil\frac{n}{2}\rceil$, $1^{th}$ copy of the  arranged subsets
$A_i=\{x_i, x_{\lceil\frac{n}{2}\rceil+i-1}, x_c\}$, denoted by the sets
$A_i^{(1)}=\{x_i^{(1)}, x_{\lceil\frac{n}{2}\rceil+i-1}^{(1)}, x_c^{(1)}\}$ of vertices  of $(C_n\Box P_k)\Box P_2$ are  doubly resolving  sets for  the arranged set  $\{x_1^{(1)},  ..., x_{nk}^{(1)}\}$ of vertices of $(C_n\Box P_k)$, and hence  the arranged sets
$A_i^{(1)}=\{x_i^{(1)}, x_{\lceil\frac{n}{2}\rceil+i-1}^{(1)}, x_c^{(1)}\}$ of vertices  of $(C_n\Box P_k)\Box P_2$ are   resolving  sets for $(C_n\Box P_k)\Box P_2$, because for each vertex in the set $\{x_1^{(2)},  ..., x_{nk}^{(2)}\}$ of vertices of $(C_n\Box P_k)\Box P_2$, we have
$$r(x_t^{(2)}|A_i^{(1)})=( d(x_t^{(1)}, x_i^{(1)})+1, d(x_t^{(1)}, x_{\lceil\frac{n}{2}\rceil+i-1}^{(1)})+1, d(x_t^{(1)}, x_{c}^{(1)})+1),$$ so   all the vertices  of $(C_n\Box P_k)\Box P_2$  have various representations with respect to the sets $A_i^{(1)}$, and hence the  minimum size of a resolving set of vertices in $(C_n\Box P_k)\Box P_2$ is $3$.
In the same way for $1\leq j\leq \lfloor\frac{n}{2}\rfloor$, if we consider $1^{th}$ copy the  arranged subsets $B_j=\{x_j, x_{\lceil\frac{n}{2}\rceil+j}, x_c\}$, denoted by the sets
$B_j^{(1)}=\{x_j^{(1)}, x_{\lceil\frac{n}{2}\rceil+j}^{(1)}, x_c^{(1)}\}$ of vertices  of $(C_n\Box P_k)\Box P_2$, where $x_c^{(1)}$ lie in the  layer $V_k^{(1)}$ of $(C_n\Box P_k)\Box P_2$ and $x_c^{(1)}$ is a compatible  vertex with  respect to $x_j^{(1)}$, then by applying the same argument
we can show that the arranged sets $B_j^{(1)}=\{x_j^{(1)}, x_{\lceil\frac{n}{2}\rceil+j}^{(1)}, x_c^{(1)}\}$ of vertices   of $(C_n\Box P_k)\Box P_2$ are  resolving sets  for $(C_n\Box P_k)\Box P_2$.
 \end{proof}

\begin{lemma}\label{f.2}
Consider the cartesian product $C_n\Box P_k$. If $n\geq3$ is an odd integer, then the  minimum size of a doubly resolving set of vertices in $(C_n\Box P_k)\Box P_2$ is greater than $3$.
\end{lemma}
\begin{proof}
Suppose $V((C_n\Box P_k)\Box P_2)=\{x_1^{(1)},  ..., x_{nk}^{(1)}\}\cup \{x_1^{(2)},  ..., x_{nk}^{(2)}\}$ and $1\leq t\leq nk$.
For $1\leq i\leq \lceil\frac{n}{2}\rceil$, let $x_i^{(1)}$ be a vertex in the layer $V_1^{(1)}$ of $(C_n\Box P_k)\Box P_2$ and $x_c^{(1)}$ be a compatible  vertex with  respect to $x_i^{(1)}$, where $x_c^{(1)}$ lie in the  layer $V_k^{(1)}$ of $(C_n\Box P_k)\Box P_2$. Based on proof of Theorem  \ref{f.1}, we  know that the arranged sets $A_i^{(1)}=\{x_i^{(1)}, x_{\lceil\frac{n}{2}\rceil+i-1}^{(1)}, x_c^{(1)}\}$ of vertices   of $(C_n\Box P_k)\Box P_2$ cannot be doubly resolving sets for $(C_n\Box P_k)\Box P_2$, because
$$r(x_t^{(2)}|A_i^{(1)})=( d(x_t^{(1)}, x_i^{(1)})+1, d(x_t^{(1)}, x_{\lceil\frac{n}{2}\rceil+i-1}^{(1)})+1, d(x_t^{(1)}, x_{c}^{(1)})+1).$$
In the same way for $1\leq j\leq \lfloor\frac{n}{2}\rfloor$,  if we consider the arranged sets $B_j^{(1)}=\{x_j^{(1)}, x_{\lceil\frac{n}{2}\rceil+j}^{(1)}, x_c^{(1)}\}$ of vertices  of $(C_n\Box P_k)\Box P_2$, where $x_c^{(1)}$ lie in the layer $V_k^{(1)}$ of $(C_n\Box P_k)\Box P_2$ and $x_c^{(1)}$ is a compatible  vertex with  respect to $x_j^{(1)}$, then we can show that the arranged sets $B_j^{(1)}$
cannot be doubly resolving sets for $(C_n\Box P_k)\Box P_2$. Hence the  minimum size of a  doubly resolving set of vertices in $(C_n\Box P_k)\Box P_2$ is  greater than $3$.
\end{proof}
\begin{theorem}\label{f.4}
Consider the cartesian product $C_n\Box P_k$. If $n\geq3$ is an odd integer, then the minimum size of a doubly resolving set of vertices in $(C_n\Box P_k)\Box P_2$ is  $4$.
\end{theorem}
\begin{proof}
Suppose $V((C_n\Box P_k)\Box P_2)=\{x_1^{(1)},  ..., x_{nk}^{(1)}\}\cup \{x_1^{(2)},  ..., x_{nk}^{(2)}\}$ and $1\leq t\leq nk$.
Based on  Theorem \ref{f.1}, we know that if $n\geq3$ is an odd integer, then $\beta((C_n\Box P_k)\Box P_2)=3$ and by Lemma \ref{f.2}, we know that the  minimum size of a  doubly resolving set of vertices in $(C_n\Box P_k)\Box P_2$ is greater than $3$. In particular,
it is well known that $\beta((C_n\Box P_k)\Box P_2)\leq \psi((C_n\Box P_k)\Box P_2)$. Now, we show that if $n\geq3$ is an odd integer, then the  minimum size of a doubly resolving set of vertices in $(C_n\Box P_k)\Box P_2$ is  $4$.
For $1\leq i\leq \lceil\frac{n}{2}\rceil$, let $x_i^{(1)}$ be a  vertex in the layer $V_1^{(1)}$ of $(C_n\Box P_k)\Box P_2$ and $x_c^{(1)}$ be a compatible  vertex with  respect to $x_i^{(1)}$, where $x_c^{(1)}$ lie in the layer $V_k^{(1)}$ of $(C_n\Box P_k)\Box P_2$. Based on Lemma \ref{f.2}, we know that the arranged sets
$A_i^{(1)}=\{x_i^{(1)}, x_{\lceil\frac{n}{2}\rceil+i-1}^{(1)}, x_c^{(1)}\}$ of vertices  of $(C_n\Box P_k)\Box P_2$, defined already,  cannot be
doubly resolving sets for  $(C_n\Box P_k)\Box P_2$. Let, $C_i=A_i^{(1)}\cup x_c^{(2)}=\{x_i^{(1)}, x_{\lceil\frac{n}{2}\rceil+i-1}^{(1)}, x_c^{(1)}, x_c^{(2)}\}$ be an arranged subset of vertices of $(C_n\Box P_k)\Box P_2$, where $x_c^{(2)}$  lie in the layer $V_k^{(2)}$ of $(C_n\Box P_k)\Box P_2$ and the vertex $x_c^{(2)}$ is adjacent to the vertex $x_c^{(1)}$. We show that the  arranged  subset $C_i$ is  a doubly resolving set of vertices in $(C_n\Box P_k)\Box P_2$.
It will be enough to show that  for any adjacent vertices $x_t^{(1)}$ and $x_t^{(2)}$, $r(x_t^{(1)}|C_i)-r(x_t^{(2)}|C_i)\neq - I$,  where $I$ indicates the unit $4$-vector $(1, ..., 1)$.
We can verify that, $r(x_t^{(1)}|A_i^{(1)})-r(x_t^{(2)}|A_i^{(1)})= - I$, where $I$ indicates the unit $3$-vector, and $r(x_t^{(1)}|x_c^{(2)})-r(x_t^{(2)}|x_c^{(2)})= 1$.  Thus the  minimum size of a doubly resolving set of vertices in $(C_n\Box P_k)\Box P_2$  is $4$.
\end{proof}
\begin{conclusion}\label{f.4.1}
Consider the cartesian product $C_n\Box P_k$. If $n\geq3$ is an odd integer, then the minimum size of a doubly resolving set of vertices in $(C_n\Box P_k)\Box P_m$ is  $4$.
\end{conclusion}
\begin{proof}
Suppose  $(C_n\Box P_k)\Box P_m$ is a graph with vertex set $\{x_1^{(1)},  ..., x_{nk}^{(1)}\}\cup \{x_1^{(2)},  ..., x_{nk}^{(2)}\}\cup ... \cup\{x_1^{(m)},  ..., x_{nk}^{(m)}\}$ so that for $1\leq t \leq nk$, the vertex $x_t^{(r)}$ is adjacent to $x_t^{(r+1)}$ in $(C_n\Box P_k)\Box P_m$, for $1\leq r\leq m-1$. On the other hand we know that for every connected graph $G$ and the path $P_m$, $\beta(G \Box P_m)\leq \beta(G)+1$. So, by considering $G=(C_n\Box P_k)$ we have $\beta(G \Box P_m)\leq \beta(G)+1=3$. Moreover it is not hard to see that
for $1\leq i\leq \lceil\frac{n}{2}\rceil$,  the arranged sets
$A_i^{(1)}=\{x_i^{(1)}, x_{\lceil\frac{n}{2}\rceil+i-1}^{(1)}, x_c^{(1)}\}$ of vertices  of $(C_n\Box P_k)\Box P_m$, defined already,
are  resolving sets with the minimum size for $(C_n\Box P_k)\Box P_m$, also by applying  the same argument in Theorem \ref{f.4}, we can see that
the arranged sets
$A_i^{(1)}=\{x_i^{(1)}, x_{\lceil\frac{n}{2}\rceil+i-1}^{(1)}, x_c^{(1)}\}$ of vertices  of $(C_n\Box P_k)\Box P_m$, cannot be doubly resolving sets  for   $(C_n\Box P_k)\Box P_m$, and hence the minimum size of a doubly resolving set of vertices in $(C_n\Box P_k)\Box P_m$ is greater than $3$.
Now, let $D_i=A_i^{(1)}\cup x_c^{(m)}=\{x_i^{(1)}, x_{\lceil\frac{n}{2}\rceil+i-1}^{(1)}, x_c^{(1)}, x_c^{(m)}\}$ be an arranged subset of vertices of $(C_n\Box P_k)\Box P_m$, where $x_c^{(m)}$  lie in the layer $V_k^{(m)}$ of $(C_n\Box P_k)\Box P_m$.
We show that the  arranged  subset $D_i$ is  a doubly resolving set of vertices in $(C_n\Box P_k)\Box P_m$.
It will be enough to show that  for every two  vertices $x_t^{(r)}$ and $x_t^{(s)}$,  $1\leq t\leq nk$, $1\leq r<s\leq m$, $r(x_t^{(r)}|D_i)-r(x_t^{(s)}|D_i)\neq -\lambda I$, where $I$ indicates the unit $4$-vector $(1, ..., 1)$ and $\lambda$ is a positive integer.
For this purpose, let the distance between two the vertices $x_t^{(r)}$ and $x_t^{(s)}$ in $(C_n\Box P_k)\Box P_m$ is $\lambda$, then
we can verify that, $r(x_t^{(r)}|A_i^{(1)})-r(x_t^{(s)}|A_i^{(1)})= - \lambda I$, where $I$ indicates the unit $3$-vector, and $r(x_t^{(t)}|x_c^{(m)})-r(x_t^{(s)}|x_c^{(m)})= \lambda$.  Thus the  minimum size of a doubly resolving set of vertices in $(C_n\Box P_k)\Box P_m$  is $4$.
\end{proof}

\begin{example}
Consider graph $(C_5\Box P_4)\Box P_4$ with vertex set $\{x_1^{(1)},  ..., x_{20}^{(1)}\}\cup \{x_1^{(2)},  ..., x_{20}^{(2)}\}\cup\{x_1^{(3)},  ..., x_{20}^{(3)}\}\cup\{x_1^{(4)},  ..., x_{20}^{(4)}\}$, we can see that the set $D_1=\{x_1^{(1)}, x_3^{(1)}, x_{16}^{(1)}, x_{16}^{(4)} \}$ of vertices of $(C_5\Box P_4)\Box P_4$ is one of the minimum doubly resolving sets for $(C_5\Box P_4)\Box P_4$, and hence the minimum size of a doubly resolving set of vertices in $(C_5\Box P_4)\Box P_4$ is $4$.\\
\\
$r(x_1^{(1)}|D_1)=(0,2,3,6),\,\,\,\,\,\,\,\,\,\,\,\,r(x_1^{(2)}|D_1)=(1,3,4,5),\,\,\,\,\,\,\,\,\,\,\,\,r(x_1^{(3)}|D_1)=(2,4,5,4),\,\,\,\,\,\,\,\,\,\,\,\,r(x_1^{(4)}|D_1)=(3,5,6,3)$\\
$r(x_2^{(1)}|D_1)=(1,1,4,7),\,\,\,\,\,\,\,\,\,\,\,\,r(x_2^{(2)}|D_1)=(2,2,5,6),\,\,\,\,\,\,\,\,\,\,\,\,r(x_2^{(3)}|D_1)=(3,3,6,5),\,\,\,\,\,\,\,\,\,\,\,\,r(x_2^{(4)}|D_1)=(4,4,7,4)$\\
$r(x_3^{(1)}|D_1)=(2,0,5,8),\,\,\,\,\,\,\,\,\,\,\,\,r(x_3^{(2)}|D_1)=(3,1,6,7),\,\,\,\,\,\,\,\,\,\,\,\,r(x_3^{(3)}|D_1)=(4,2,7,6),\,\,\,\,\,\,\,\,\,\,\,\,r(x_3^{(4)}|D_1)=(5,3,8,5)$\\
$r(x_4^{(1)}|D_1)=(2,1,5,8),\,\,\,\,\,\,\,\,\,\,\,\,r(x_4^{(2)}|D_1)=(3,2,6,7),\,\,\,\,\,\,\,\,\,\,\,\,r(x_4^{(3)}|D_1)=(4,3,7,6),\,\,\,\,\,\,\,\,\,\,\,\,r(x_4^{(4)}|D_1)=(5,4,8,5)$\\
$r(x_5^{(1)}|D_1)=(1,2,4,7),\,\,\,\,\,\,\,\,\,\,\,\,r(x_5^{(2)}|D_1)=(2,3,5,6),\,\,\,\,\,\,\,\,\,\,\,\,r(x_5^{(3)}|D_1)=(3,4,6,5),\,\,\,\,\,\,\,\,\,\,\,\,r(x_5^{(4)}|D_1)=(4,5,7,4)$\\
\\
$r(x_6^{(1)}|D_1)=(1,3,2,5),\,\,\,\,\,\,\,\,\,\,\,\,r(x_6^{(2)}|D_1)=(2,4,3,4),\,\,\,\,\,\,\,\,\,\,\,\,r(x_6^{(3)}|D_1)=(3,5,4,3),\,\,\,\,\,\,\,\,\,\,\,\,r(x_6^{(4)}|D_1)=(4,6,5,2)$\\
$r(x_7^{(1)}|D_1)=(2,2,3,6),\,\,\,\,\,\,\,\,\,\,\,\,r(x_7^{(2)}|D_1)=(3,3,4,5),\,\,\,\,\,\,\,\,\,\,\,\,r(x_7^{(3)}|D_1)=(4,4,5,4),\,\,\,\,\,\,\,\,\,\,\,\,r(x_7^{(4)}|D_1)=(5,5,6,3)$\\
$r(x_8^{(1)}|D_1)=(3,1,4,7),\,\,\,\,\,\,\,\,\,\,\,\,r(x_8^{(2)}|D_1)=(4,2,5,6),\,\,\,\,\,\,\,\,\,\,\,\,r(x_8^{(3)}|D_1)=(5,3,6,5),\,\,\,\,\,\,\,\,\,\,\,\,r(x_8^{(4)}|D_1)=(6,4,7,4)$\\
$r(x_9^{(1)}|D_1)=(3,2,4,7),\,\,\,\,\,\,\,\,\,\,\,\,r(x_9^{(2)}|D_1)=(4,3,5,6),\,\,\,\,\,\,\,\,\,\,\,\,r(x_9^{(3)}|D_1)=(5,4,6,5),\,\,\,\,\,\,\,\,\,\,\,\,r(x_9^{(4)}|D_1)=(6,5,7,4)$\\
$r(x_{10}^{(1)}|D_1)=(2,3,3,6),\,\,\,\,\,\,\,\,\,\,\,\,r(x_{10}^{(2)}|D_1)=(3,4,4,5),\,\,\,\,\,\,\,\,\,\,\,\,r(x_{10}^{(3)}|D_1)=(4,5,5,4),\,\,\,\,\,\,\,\,\,\,\,\,r(x_{10}^{(4)}|D_1)=(5,6,6,3)$\\
\\
$r(x_{11}^{(1)}|D_1)=(2,4,1,4),\,\,\,\,\,\,\,\,\,\,\,\,r(x_{11}^{(2)}|D_1)=(3,5,2,3),\,\,\,\,\,\,\,\,\,\,\,\,r(x_{11}^{(3)}|D_1)=(4,6,3,4),\,\,\,\,\,\,\,\,\,\,\,\,r(x_{11}^{(4)}|D_1)=(5,7,4,3)$\\
$r(x_{12}^{(1)}|D_1)=(3,3,2,5),\,\,\,\,\,\,\,\,\,\,\,\,r(x_{12}^{(2)}|D_1)=(4,4,3,4),\,\,\,\,\,\,\,\,\,\,\,\,r(x_{12}^{(3)}|D_1)=(5,5,4,3),\,\,\,\,\,\,\,\,\,\,\,\,r(x_{12}^{(4)}|D_1)=(6,6,5,2)$\\
$r(x_{13}^{(1)}|D_1)=(4,2,3,6),\,\,\,\,\,\,\,\,\,\,\,\,r(x_{13}^{(2)}|D_1)=(5,3,4,5),\,\,\,\,\,\,\,\,\,\,\,\,r(x_{13}^{(3)}|D_1)=(6,4,5,4),\,\,\,\,\,\,\,\,\,\,\,\,r(x_{13}^{(4)}|D_1)=(7,5,6,3)$\\
$r(x_{14}^{(1)}|D_1)=(4,3,3,6),\,\,\,\,\,\,\,\,\,\,\,\,r(x_{14}^{(2)}|D_1)=(5,4,4,5),\,\,\,\,\,\,\,\,\,\,\,\,r(x_{14}^{(3)}|D_1)=(6,5,5,4),\,\,\,\,\,\,\,\,\,\,\,\,r(x_{14}^{(4)}|D_1)=(7,6,6,3)$\\
$r(x_{15}^{(1)}|D_1)=(3,4,2,5),\,\,\,\,\,\,\,\,\,\,\,\,r(x_{15}^{(2)}|D_1)=(4,5,3,4),\,\,\,\,\,\,\,\,\,\,\,\,r(x_{15}^{(3)}|D_1)=(5,6,4,3),\,\,\,\,\,\,\,\,\,\,\,\,r(x_{15}^{(4)}|D_1)=(6,7,5,2)$\\
\\
\\
$r(x_{16}^{(1)}|D_1)=(3,5,0,3),\,\,\,\,\,\,\,\,\,\,\,\,r(x_{16}^{(2)}|D_1)=(4,6,1,2),\,\,\,\,\,\,\,\,\,\,\,\,r(x_{16}^{(3)}|D_1)=(5,7,2,1),\,\,\,\,\,\,\,\,\,\,\,\,r(x_{16}^{(4)}|D_1)=(6,8,3,0)$\\
$r(x_{17}^{(1)}|D_1)=(4,4,1,4),\,\,\,\,\,\,\,\,\,\,\,\,r(x_{17}^{(2)}|D_1)=(5,5,2,3),\,\,\,\,\,\,\,\,\,\,\,\,r(x_{17}^{(3)}|D_1)=(6,6,3,2),\,\,\,\,\,\,\,\,\,\,\,\,r(x_{17}^{(4)}|D_1)=(7,7,4,1)$\\
$r(x_{18}^{(1)}|D_1)=(5,3,2,5),\,\,\,\,\,\,\,\,\,\,\,\,r(x_{18}^{(2)}|D_1)=(6,4,3,4),\,\,\,\,\,\,\,\,\,\,\,\,r(x_{18}^{(3)}|D_1)=(7,5,4,3),\,\,\,\,\,\,\,\,\,\,\,\,r(x_{18}^{(4)}|D_1)=(8,6,5,2)$\\
$r(x_{19}^{(1)}|D_1)=(5,4,2,5),\,\,\,\,\,\,\,\,\,\,\,\,r(x_{19}^{(2)}|D_1)=(6,5,3,4),\,\,\,\,\,\,\,\,\,\,\,\,r(x_{19}^{(3)}|D_1)=(7,6,4,3),\,\,\,\,\,\,\,\,\,\,\,\,r(x_{19}^{(4)}|D_1)=(8,7,5,2)$\\
$r(x_{20}^{(1)}|D_1)=(4,5,1,4),\,\,\,\,\,\,\,\,\,\,\,\,r(x_{20}^{(2)}|D_1)=(5,6,2,3),\,\,\,\,\,\,\,\,\,\,\,\,r(x_{20}^{(3)}|D_1)=(6,7,3,2),\,\,\,\,\,\,\,\,\,\,\,\,r(x_{20}^{(4)}|D_1)=(7,8,4,1)$\\
\end{example}
\begin{remark}\label{u.1}
Consider the cartesian product $C_n\Box P_k$. If $n\geq4$ is an even integer,
 then every pair of various vertices in  $C_n\Box P_k$ cannot be a resolving set for $C_n\Box P_k$.
\end{remark}
\begin{lemma}\label{u.2}
Consider the cartesian product $C_n\Box P_k$. If $n\geq 4$ is an even integer, then the minimum size of a doubly resolving set of vertices for the cartesian product $C_n\Box P_k$ is $4$.
\end{lemma}
\begin{proof}
In the following cases we show that the minimum size of a doubly resolving set of vertices for the cartesian product $C_n\Box P_k$ is $4$.\\

Case 1. Contrary to Theorem \ref{d.1}, if $n\geq 4$ is an even integer then every minimum resolving set of  vertices in $C_n\Box P_k$ may be lie in  congruous layers $V_1$, $V_k$, or $V_1\cup V_k$.
Also, based on  Theorem \ref{b.7}, we know that the  minimum size of a resolving set  in the cartesian product $C_n\Box P_k$ is  $3$.
If $E_1$ is an arranged subset of  vertices of $C_n\Box P_k$  so that  $E_1$ is a minimal resolving set for $C_n\Box P_k$ and two elements of $E_1$  lie in  the  layer $V_1$  and one element of $E_1$  lie in  the  layer $V_k$, then without loss of generality we can consider  $E_1=\{x_1, x_2, x_c \}$, where $x_c$ is a compatible  vertex with  respect to $x_1$ and $x_c$ lie in the layer $V_k$ of $C_n\Box P_k$.
Besides, the arranged subset $E_1=\{x_1, x_2, x_c \}$  of vertices of  $C_n\Box P_k$  cannot be a doubly resolving set  for $C_n\Box P_k$, because $n$ is an even integer and each layer $V_p$, $1\leq p\leq k$ of  $C_n\Box P_k$ is isomorphic to the cycle $C_n$, and hence there are two adjacent vertices  in each layer $V_p$ of  $C_n\Box P_k$  say $x_i$ and $x_j$, $i<j$, so that $r(x_i|E_1)-r(x_j|E_1)=-I$, where $I$ indicates the unit $3$-vector $(1, 1, 1)$.
Thus the arranged  subset $E_1=\{x_1, x_2, x_c \}$  cannot be a doubly resolving set for $C_n\Box P_k$.
If $E_2$ is an arranged subset of  vertices of $C_n\Box P_k$  so that  $E_2$ is a minimal resolving set for $C_n\Box P_k$ and  all the elements of $E_2$  lie in  exactly one of the congruous layers $V_1$ or $V_k$, then without loss of generality we can consider  $E_2=\{x_1, x_{\frac{n}{2}}, x_{{\frac{n}{2}}+1} \}$. Besides, the arranged subset $E_2=\{x_1, x_{\frac{n}{2}}, x_{{\frac{n}{2}}+1} \}$  of vertices in the layer $V_1$ of the cartesian product $C_n\Box P_k$  cannot be a doubly resolving set  for $C_n\Box P_k$, because if we consider two compatible vertices $x_e\in V_p$ and $x_d\in V_q$ in $C_n\Box P_k$,  $1\leq p< q\leq k$, then there is a positive  integer $\lambda$ so that $r(x_e|E_2)-r(x_d|E_2)=-\lambda I$, where $I$ indicates the unit $3$-vector $(1, 1, 1)$, and hence the arranged  subset $E_2=\{x_1, x_{\frac{n}{2}}, x_{{\frac{n}{2}}+1} \}$  cannot be a doubly resolving set for $C_n\Box P_k$. Thus  the minimum size of a doubly resolving set of vertices for the cartesian product $C_n\Box P_k$ is greater than $3$.\\

Case 2. Now, we show that the minimum size of a doubly resolving set of vertices in  the cartesian product $C_n\Box P_k$ is $4$.
Let $x_c$ be a compatible vertex with  respect to $x_1$, where $x_c$ lie in the layer $V_k$ of $C_n\Box P_k$ then we can show that
the arranged subset $E_3=E_2\cup x_c=\{x_1, x_{\frac{n}{2}}, x_{{\frac{n}{2}}+1}, x_c\}$ of vertices in  the cartesian product $C_n\Box P_k$ is
one of the minimum  doubly resolving sets for the cartesian product $C_n\Box P_k$. It will be enough to show that for any compatible vertices $x_e$
and $x_d$ in $C_n\Box P_k$, $r(x_e|E_3)-r(x_d|E_3)\neq\lambda I$. Suppose  $x_e\in V_p$ and $x_d\in V_q$ are compatible vertices in the
cartesian product $C_n\Box P_k$, $1\leq p<q\leq k$. Hence, $r(x_e|E_2)-r(x_d|E_2)=-\lambda I$, where $\lambda$ is a positive integer, and $I$
indicates the unit $3$-vector $(1,1, 1)$. Also, for the compatible vertex $x_c$ with  respect to $x_1$,  $r(x_e|x_c)-r(x_d|x_c)=\lambda$.
 So, $r(x_e|E_3)-r(x_d|E_3)\neq\lambda I$,  where $I$ indicates the unit $4$-vector $(1,..., 1)$.\newline
\end{proof}
\begin{theorem}\label{u.3}
Consider the cartesian product $C_n\Box P_k$. If $n\geq4$ is an even integer,
 then the minimum size of a  resolving set of vertices in $(C_n\Box P_k)\Box P_m$ is  $4$.
\end{theorem}
\begin{proof}
Suppose  $(C_n\Box P_k)\Box P_m$ is a graph with vertex set $\{x_1^{(1)},  ..., x_{nk}^{(1)}\}\cup \{x_1^{(2)},  ..., x_{nk}^{(2)}\}\cup ... \cup\{x_1^{(m)},  ..., x_{nk}^{(m)}\}$ so that for $1\leq t \leq nk$, the vertex $x_t^{(r)}$ is adjacent to $x_t^{(r+1)}$ in $(C_n\Box P_k)\Box P_m$, for $1\leq r\leq m-1$.
Hence, none of the minimal resolving sets of $C_n\Box P_k$ cannot be a resolving set for $(C_n\Box P_k)\Box P_m$. Therefore, the  minimum size of a resolving set of vertices in $(C_n\Box P_k)\Box P_m$ must be greater than $3$.
Now, we show that the  minimum size of a resolving set of vertices in $(C_n\Box P_k)\Box P_m$ is  $4$.
Let $x_1^{(1)}$ be a  vertex in the layer $V_1^{(1)}$ of $(C_n\Box P_k)\Box P_m$ and $x_c^{(1)}$ be a compatible  vertex with  respect to $x_1^{(1)}$, where $x_c^{(1)}$ lie in the layer $V_k^{(1)}$ of $(C_n\Box P_k)\Box P_m$. Based on Lemma \ref{u.2}, we know that
$1^{th}$ copy of the  arranged subset $E_3=\{x_1, x_{\frac{n}{2}}, x_{{\frac{n}{2}}+1}, x_c\}$, denoted by the set
$E_3^{(1)}=\{x_1^{(1)}, x_{\frac{n}{2}}^{(1)}, x_{{\frac{n}{2}}+1}^{(1)}, x_c^{(1)}\}$ is one of the minimum  doubly resolving sets for the cartesian product $C_n\Box P_k$. Besides, the vertex $x_t^{(r)}$ is adjacent to the vertex $x_t^{(r+1)}$ in $(C_n\Box P_k)\Box P_m$,
and hence  the arranged set
$E_3^{(1)}=\{x_1^{(1)}, x_{\frac{n}{2}}^{(1)}, x_{{\frac{n}{2}}+1}^{(1)}, x_c^{(1)}\}$ of vertices  of $(C_n\Box P_k)\Box P_m$ is one of  the  resolving  sets for $(C_n\Box P_k)\Box P_m$, because for each vertex $x_t^{(i)}$ of $(C_n\Box P_k)\Box P_m$, we have
$$r(x_t^{(i)}|E_3^{(1)})=( d(x_t^{(1)}, x_1^{(1)})+i-1, d(x_t^{(1)}, x_{\frac{n}{2}}^{(1)})+i-1, d(x_t^{(1)}, x_{{\frac{n}{2}}+1}^{(1)})+i-1, d(x_t^{(1)}, x_{c}^{(1)})+i-1),$$ so   all the vertices $\{x_1^{(1)},  ..., x_{nk}^{(1)}\}\cup ...\cup \{x_1^{(m)},  ..., x_{nk}^{(m)}\}$ of $(C_n\Box P_k)\Box P_m$  have various representations with respect to the set $E_3^{(1)}$. Thus the  minimum size of a resolving set of vertices in $C_n\Box P_k\Box P_m$ is $4$.
\end{proof}
\begin{theorem}\label{u.4}
Consider the cartesian product $C_n\Box P_k$. If $n\geq4$ is an even integer,
 then the minimum size of a doubly resolving set of vertices in $(C_n\Box P_k)\Box P_m$ is  $5$.
\end{theorem}
\begin{proof}
Suppose  $(C_n\Box P_k)\Box P_m$ is a graph with vertex set $\{x_1^{(1)},  ..., x_{nk}^{(1)}\}\cup \{x_1^{(2)},  ..., x_{nk}^{(2)}\}\cup ... \cup\{x_1^{(m)},  ..., x_{nk}^{(m)}\}$ so that for $1\leq t \leq nk$, the vertex $x_t^{(r)}$ is adjacent to $x_t^{(r+1)}$ in $(C_n\Box P_k)\Box P_m$, for $1\leq r\leq m-1$. Based on the previous Theorem, we know that
the arranged set $E_3^{(1)}=\{x_1^{(1)}, x_{\frac{n}{2}}^{(1)}, x_{{\frac{n}{2}}+1}^{(1)}, x_c^{(1)}\}$ of vertices  of $(C_n\Box P_k)\Box P_m$ is one of  the  resolving  sets for $(C_n\Box P_k)\Box P_m$, so that the arranged set $E_3^{(1)}=\{x_1^{(1)}, x_{\frac{n}{2}}^{(1)}, x_{{\frac{n}{2}}+1}^{(1)}, x_c^{(1)}\}$ of vertices  of $(C_n\Box P_k)\Box P_m$ cannot be a doubly resolving set for $(C_n\Box P_k)\Box P_m$, and hence the minimum size of a doubly resolving set of vertices in $(C_n\Box P_k)\Box P_m$ is greater than $4$.
Now, let $E_4=E_3^{(1)}\cup x_c^{(m)}=\{x_1^{(1)}, x_{\frac{n}{2}}^{(1)}, x_{{\frac{n}{2}}+1}^{(1)}, x_c^{(1)}, x_c^{(m)}\}$
be an arranged subset of vertices of $(C_n\Box P_k)\Box P_m$, where $x_c^{(m)}$  lie in the layer $V_k^{(m)}$ of $(C_n\Box P_k)\Box P_m$.
It will be enough to show that  for every two  vertices $x_t^{(r)}$ and $x_t^{(s)}$,  $1\leq t\leq nk$, $1\leq r<s\leq m$, $r(x_t^{(r)}|E_4)-r(x_t^{(s)}|E_4)\neq -\lambda I$, where $I$ indicates the unit $5$-vector $(1, ..., 1)$ and $\lambda$ is a positive integer.
For this purpose, let the distance between two the vertices $x_t^{(r)}$ and $x_t^{(s)}$ in $(C_n\Box P_k)\Box P_m$ is $\lambda$, then
we can verify that, $r(x_t^{(r)}|E_3^{(1)})-r(x_t^{(s)}|E_3^{(1)})= - \lambda I$, where $I$ indicates the unit $4$-vector, and $r(x_t^{(t)}|x_c^{(m)})-r(x_t^{(s)}|x_c^{(m)})= \lambda$. Therefore, the  arranged  subset $E_4$ is one of  the minimum  doubly resolving sets of vertices in $(C_n\Box P_k)\Box P_m$. Thus the  minimum size of a doubly resolving set of vertices in $(C_n\Box P_k)\Box P_m$  is $5$.
\end{proof}
\begin{example}
Consider graph $(C_4\Box P_3)\Box P_4$ with vertex set $\{x_1^{(1)},  ..., x_{12}^{(1)}\}\cup \{x_1^{(2)},  ..., x_{12}^{(2)}\}\cup\{x_1^{(3)},  ..., x_{12}^{(3)}\}\cup\{x_1^{(4)},  ..., x_{12}^{(4)}\}$, we can see that the set $E=\{x_1^{(1)}, x_2^{(1)} x_3^{(1)}, x_{9}^{(1)}, x_{9}^{(4)} \}$ of vertices of $(C_4\Box P_3)\Box P_4$ is one of the minimum doubly resolving sets for $(C_4\Box P_3)\Box P_4$, and hence the minimum size of a doubly resolving set of vertices in $(C_4\Box P_3)\Box P_4$ is $4$.\\
\\
$r(x_1^{(1)}|E)=(0,1,2,2,5),\,\,\,\,\,\,\,\,\,\,\,\,r(x_1^{(2)}|E)=(1,2,3,3,4),\,\,\,\,\,\,\,\,\,\,\,\,r(x_1^{(3)}|E)=(2,3,4,4,3),\,\,\,\,\,\,\,\,\,\,\,\,r(x_1^{(4)}|E)=(3,4,5,5,2)$\\
$r(x_2^{(1)}|E)=(1,0,1,3,6),\,\,\,\,\,\,\,\,\,\,\,\,r(x_2^{(2)}|E)=(2,1,2,4,5),\,\,\,\,\,\,\,\,\,\,\,\,r(x_2^{(3)}|E)=(3,2,3,5,4),\,\,\,\,\,\,\,\,\,\,\,\,r(x_2^{(4)}|E)=(4,3,4,6,3)$\\
$r(x_3^{(1)}|E)=(2,1,0,4,7),\,\,\,\,\,\,\,\,\,\,\,\,r(x_3^{(2)}|E)=(3,2,1,5,6),\,\,\,\,\,\,\,\,\,\,\,\,r(x_3^{(3)}|E)=(4,3,2,6,5),\,\,\,\,\,\,\,\,\,\,\,\,r(x_3^{(4)}|E)=(5,4,3,7,4)$\\
$r(x_4^{(1)}|E)=(1,2,1,3,6),\,\,\,\,\,\,\,\,\,\,\,\,r(x_4^{(2)}|E)=(2,3,2,4,5),\,\,\,\,\,\,\,\,\,\,\,\,r(x_4^{(3)}|E)=(3,4,3,5,4),\,\,\,\,\,\,\,\,\,\,\,\,r(x_4^{(4)}|E)=(4,5,4,6,3)$\\
\\
$r(x_5^{(1)}|E)=(1,2,3,1,4),\,\,\,\,\,\,\,\,\,\,\,\,r(x_5^{(2)}|E)=(2,3,4,2,3),\,\,\,\,\,\,\,\,\,\,\,\,r(x_5^{(3)}|E)=(3,4,5,3,2),\,\,\,\,\,\,\,\,\,\,\,\,r(x_5^{(4)}|E)=(4,5,6,4,1)$\\
$r(x_6^{(1)}|E)=(2,1,2,2,5),\,\,\,\,\,\,\,\,\,\,\,\,r(x_6^{(2)}|E)=(3,2,3,3,4),\,\,\,\,\,\,\,\,\,\,\,\,r(x_6^{(3)}|E)=(4,3,4,4,3),\,\,\,\,\,\,\,\,\,\,\,\,r(x_6^{(4)}|E)=(5,4,5,5,2)$\\
$r(x_7^{(1)}|E)=(3,2,1,3,6),\,\,\,\,\,\,\,\,\,\,\,\,r(x_7^{(2)}|E)=(4,3,2,4,5),\,\,\,\,\,\,\,\,\,\,\,\,r(x_7^{(3)}|E)=(5,4,3,5,4),\,\,\,\,\,\,\,\,\,\,\,\,r(x_7^{(4)}|E)=(6,5,4,6,3)$\\
$r(x_8^{(1)}|E)=(2,3,2,2,5),\,\,\,\,\,\,\,\,\,\,\,\,r(x_8^{(2)}|E)=(3,4,3,3,4),\,\,\,\,\,\,\,\,\,\,\,\,r(x_8^{(3)}|E)=(4,5,4,4,3),\,\,\,\,\,\,\,\,\,\,\,\,r(x_8^{(4)}|E)=(5,6,5,5,2)$\\
\\
$r(x_9^{(1)}|E)=(2,3,4,0,3),\,\,\,\,\,\,\,\,\,\,\,\,r(x_9^{(2)}|E)=(3,4,5,1,2),\,\,\,\,\,\,\,\,\,\,\,\,r(x_9^{(3)}|E)=(4,5,6,2,1),\,\,\,\,\,\,\,\,\,\,\,\,r(x_9^{(4)}|E)=(5,6,7,3,0)$\\
$r(x_{10}^{(1)}|E)=(3,2,3,1,4),\,\,\,\,\,\,\,\,\,\,\,\,r(x_{10}^{(2)}|E)=(4,3,4,2,3),\,\,\,\,\,\,\,\,\,\,\,\,r(x_{10}^{(3)}|E)=(5,4,5,3,2),\,\,\,\,\,\,\,\,\,\,\,\,r(x_{10}^{(4)}|E)=(6,5,6,4,1)$\\
$r(x_{11}^{(1)}|E)=(4,3,2,2,5),\,\,\,\,\,\,\,\,\,\,\,\,r(x_{11}^{(2)}|E)=(5,4,3,3,4),\,\,\,\,\,\,\,\,\,\,\,\,r(x_{11}^{(3)}|E)=(6,5,4,4,3),\,\,\,\,\,\,\,\,\,\,\,\,r(x_{11}^{(4)}|E)=(7,6,5,5,2)$\\
$r(x_{12}^{(1)}|E)=(3,4,3,1,4),\,\,\,\,\,\,\,\,\,\,\,\,r(x_{12}^{(2)}|E)=(4,5,4,2,3),\,\,\,\,\,\,\,\,\,\,\,\,r(x_{12}^{(3)}|E)=(5,6,5,3,2),\,\,\,\,\,\,\,\,\,\,\,\,r(x_{12}^{(4)}|E)=(6,7,6,4,1)$\\
\end{example}
\begin{theorem}\label{w.1}
If $n$ is an even or odd integer is greater than or equal to $3$, then  the minimum size of a strong resolving set  of vertices for the cartesian product $(C_n\Box P_k)\Box P_m$ is  $2n$.
\end{theorem}
\begin{proof}
 Suppose that $(C_n\Box P_k)\Box P_m$ is a graph with vertex set $\{x_1^{(1)},  ..., x_{nk}^{(1)}\}\cup \{x_1^{(2)},  ..., x_{nk}^{(2)}\}\cup ... \cup\{x_1^{(m)},  ..., x_{nk}^{(m)}\}$ so that for $1\leq t \leq nk$, the vertex $x_t^{(r)}$ is adjacent to $x_t^{(r+1)}$ in $(C_n\Box P_k)\Box P_m$, for $1\leq r\leq m-1$.  We know that each  vertex of  the layer  $V_1^{(1)}$  is maximally distant from a vertex of  the layer $V_k^{(m)}$ and
each  vertex of  the layer  $V_k^{(m)}$  is maximally distant from a vertex of  the layer $V_1^{(1)}$. In particular, each  vertex of  the layer  $V_1^{(m)}$  is maximally distant from a vertex of  the layer $V_k^{(1)}$ and
each  vertex of  the layer  $V_k^{(1)}$  is maximally distant from a vertex of  the layer $V_1^{(m)}$, and hence the minimum size of a strong resolving set of vertices for the cartesian product $(C_n\Box P_k)\Box P_m$ is equal or greater than $2n$, because it is well known that for
every pair of mutually maximally distant vertices $u$ and $v$ of a connected graph $G$ and for every strong metric
basis $S$ of $G$, it follows that $u\in S$ or $v\in S$.
Suppose the set $\{x_1^{(1)},  ...,  x_n^{(1)}\}$ is an arranged subset of vertices in  the layer $V_1^{(1)}$ of the cartesian product $(C_n\Box P_k)\Box P_m$ and suppose that the set $\{x_1^{(m)},  ...,  x_n^{(m)}\}$ is an arranged subset of vertices in  the layer $V_1^{(m)}$ of the cartesian product $(C_n\Box P_k)\Box P_m$. Now, let $T=\{x_1^{(1)},  ...,  x_n^{(1)}\}\cup \{x_1^{(m)},  ...,  x_n^{(m)}\}$ be an arranged subset of vertices   of the cartesian product $(C_n\Box P_k)\Box P_m$. In the following cases we show that the arranged set $T$, defined already, is one of the  minimum strong resolving sets  of vertices for the cartesian product $(C_n\Box P_k)\Box P_m$. For this purpose let $x_e^{(i)}$ and $x_d^{(j)}$ be two various vertices of $(C_n\Box P_k)\Box P_m$, $1\leq i,j\leq m$, $1\leq e,d\leq nk$ and $1\leq r\leq n$.\\

Case 1. If $i=j$ then $x_e^{(i)}$ and $x_d^{(i)}$ lie in $i^{th}$ copy of $(C_n\Box P_k)$ with vertex set $\{x_1^{(i)},  ..., x_{nk}^{(i)}\}$ so that $i^{th}$ copy of $(C_n\Box P_k)$ is a subgraph of $(C_n\Box P_k)\Box P_m$. Since $i=j$ then we can assume that $e<d$, because $x_e^{(i)}$ and $x_d^{(i)}$ are various vertices.\\

Case 1.1. If both vertices $x_e^{(i)}$ and $x_d^{(i)}$ are  compatible in $i^{th}$ copy of $(C_n\Box P_k)$ relative to $x_r^{(i)}\in V_1^{(i)}$, then there is the vertex $x_r^{(1)}\in V_1^{(1)}\subset T$ so that $x_e^{(i)}$ belongs to shortest path  $x_r^{(1)}-x_d^{(i)}$, say as $x_r^{(1)},..., x_r^{(i)},..., x_e^{(i)}, ..., x_d^{(i)}$.\\

Case 1.2. Suppose both vertices $x_e^{(i)}$ and $x_d^{(i)}$ are not  compatible in $i^{th}$ copy of $(C_n\Box P_k)$, and lie
in various layers  or lie in the same layer in $i^{th}$ copy of $(C_n\Box P_k)$, also let
$x_r^{(i)}\in V_1^{(i)}$, be a compatible vertex relative to $x_e^{(i)}$. Hence there is
 the vertex $x_r^{(m)}\in V_1^{(m)}\subset T$ so that
  $x_e^{(i)}$ belongs to shortest path  $x_r^{(m)}-x_d^{(i)}$,
 say as $x_r^{(m)},..., x_r^{(i)},..., x_e^{(i)}, ..., x_d^{(i)}$.\\

Case 2. If $i\neq j$ then $x_e^{(i)}$ lie in $i^{th}$ copy of $(C_n\Box P_k)$ with vertex set $\{x_1^{(i)},  ..., x_{nk}^{(i)}\}$, and $x_d^{(j)}$ lie in $j^{th}$ copy of $(C_n\Box P_k)$ with vertex set $\{x_1^{(j)},  ..., x_{nk}^{(j)}\}$. In this case we can assume that $i<j$.\\

Case 2.1. If $e=d$ and $x_r^{(j)}\in V_1^{(j)}$  is a compatible vertex relative to  $x_d^{(j)}$, then there is the vertex $x_r^{(m)}\in V_1^{(m)}\subset T$ so that $x_d^{(j)}$ belongs to shortest path  $x_r^{(m)}-x_e^{(i)}$, say as $x_r^{(m)},..., x_r^{(j)},..., x_d^{(j)}, ..., x_e^{(i)}$.\\

Case 2.2. If $e< d$, also $x_e^{(i)}$ and $x_d^{(j)}$ lie in various layers of $(C_n\Box P_k)\Box P_m$ or $x_e^{(i)}$ and $x_d^{(j)}$ lie in the same layer of $(C_n\Box P_k)\Box P_m$ and $x_r^{(i)}\in V_1^{(i)}$ is a compatible vertex relative to  $x_e^{(i)}$, then there
is the vertex $x_r^{(1)}\in V_1^{(1)}\subset T$ so that  $x_e^{(i)}$ belongs to shortest path  $x_r^{(1)}-x_d^{(j)}$, say as $x_r^{(1)},..., x_r^{(i)},..., x_e^{(i)}, ..., x_d^{(j)}$.\\

Case 2.3. If $e> d$, also $x_e^{(i)}$ and $x_d^{(j)}$ lie in various layers of $(C_n\Box P_k)\Box P_m$ or $x_e^{(i)}$ and $x_d^{(j)}$ lie in the same layer of $(C_n\Box P_k)\Box P_m$ and $x_r^{(j)}\in V_1^{(j)}$ is a compatible vertex relative to  $x_d^{(j)}$, then there
is the vertex $x_r^{(m)}\in V_1^{(m)}\subset T$ so that  $x_d^{(j)}$ belongs to shortest path  $x_r^{(m)}-x_e^{(i)}$, say as
 $x_r^{(m)},..., x_r^{(j)},..., x_d^{(j)}, ..., x_e^{(i)}$.

\end{proof}
\bigskip
{\footnotesize

\noindent \textbf{Data Availability}\\
No data were used to support this study.\\[2mm]
\noindent \textbf{Conflicts of Interest}\\
The authors declare that there are no conflicts of interest
regarding the publication of this paper.\\[2mm]
\noindent \textbf{Acknowledgements}\\
This work was supported in part by Anhui Provincial Natural Science Foundation under Grant 2008085J01 and Natural Science Fund of Education Department of Anhui Province under Grant KJ2020A0478. \\[2mm]
\noindent \textbf{Authors' informations}\\
\noindent Jia-Bao Liu${}^a$
(\url{liujiabaoad@163.com;liujiabao@ahjzu.edu.cn})\\
Ali Zafari${}^{a}$(\textsc{Corresponding Author})
(\url{zafari.math.pu@gmail.com}; \url{zafari.math@pnu.ac.ir})\\

\noindent ${}^{a}$ School of Mathematics and Physics, Anhui Jianzhu University, Hefei 230601, P.R. China.\\
${}^{a}$ Department of Mathematics, Faculty of Science,
Payame Noor University, P.O. Box 19395-4697, Tehran, Iran.

{\footnotesize


\bigskip
\end{document}